\documentclass[12pt,leqno,draft]{article}
\usepackage{amsfonts}
\usepackage{amsmath, amsthm, amsfonts, amssymb, color}
\usepackage{mathrsfs}
\usepackage{color}
\newtheorem{thm}{Theorem}[section]

\newtheorem{lem}[thm]{Lemma}

\newtheorem{exa}[thm]{Example}
\newtheorem{rem}[thm]{Remark}
\theoremstyle{definition}

\newcommand{\scr}[1]{\mathscr #1}
\definecolor{wco}{rgb}{0.5,0.2,0.3}

\numberwithin{equation}{section} \theoremstyle{remark}

\newcommand{\ua}{\uparrow}

\title{{\bf 
 Exponential Ergodicity for McKean-Vlasov SDEs with Singular Interactions}\footnote{Supported in
 part by  National Key R\&D Program of China (2022YFA1006000), NNSFC (12271398, 12531007). } }
\author{
{\bf   Xing Huang, Feng-Yu Wang }\\
\footnotesize{Center for Applied Mathematics and KL-AAGDM, Tianjin
University, Tianjin 300072, China}\\
\footnotesize{  xinghuang@tju.edu.cn, 
wangfy@tju.edu.cn}\\
}
\begin{document}
\allowdisplaybreaks
\def\R{\mathbb R}  \def\ff{\frac} \def\ss{\sqrt} \def\B{\mathbf
B} \def\W{\mathbb W}
\def\N{\mathbb N} \def\kk{\kappa} \def\m{{\bf m}}
\def\ee{\varepsilon}\def\ddd{D^*}
\def\dd{\delta} \def\DD{\Delta} \def\vv{\varepsilon} \def\rr{\rho}
\def\<{\langle} \def\>{\rangle} \def\gg{\gamma}
  \def\nn{\nabla} \def\pp{\partial} \def\E{\mathbb E}
\def\d{\text{\rm{d}}} \def\bb{\beta} \def\aa{\alpha} \def\D{\scr D}
  \def\si{\sigma} \def\ess{\text{\rm{ess}}}
\def\beg{\begin} \def\beq{\begin{equation}}  \def\F{\scr F}
\def\Ric{\text{\rm{Ric}}} \def\Hess{\text{\rm{Hess}}}
\def\e{\text{\rm{e}}} \def\ua{\underline a} \def\OO{\Omega}  \def\oo{\omega}
 \def\tt{\tilde} \def\Ric{\text{\rm{Ric}}}
\def\cut{\text{\rm{cut}}} \def\P{\mathbb P} \def\ifn{I_n(f^{\bigotimes n})}
\def\C{\scr C}      \def\aaa{\mathbf{r}}     \def\r{r}
\def\gap{\text{\rm{gap}}} \def\prr{\pi_{{\bf m},\varrho}}  \def\r{\mathbf r}
\def\Z{\mathbb Z} \def\vrr{\varrho} \def\ll{\lambda}
\def\L{\scr L}\def\Tt{\tt} \def\TT{\tt}\def\II{\mathbb I}
\def\i{{\rm in}}\def\Sect{{\rm Sect}}  \def\H{\mathbb H}
\def\M{\scr M}\def\Q{\mathbb Q} \def\texto{\text{o}} \def\LL{\Lambda}
\def\Rank{{\rm Rank}} \def\B{\scr B} \def\i{{\rm i}} \def\HR{\hat{\R}^d}
\def\to{\rightarrow}\def\l{\ell}\def\iint{\int}
\def\EE{\scr E}\def\Cut{{\rm Cut}}
\def\A{\scr A} \def\Lip{{\rm Lip}}
\def\BB{\scr B}\def\Ent{{\rm Ent}}\def\L{\scr L}
\def\R{\mathbb R}  \def\ff{\frac} \def\ss{\sqrt} \def\B{\mathbf
B}
\def\N{\mathbb N} \def\kk{\kappa} \def\m{{\bf m}}
\def\dd{\delta} \def\DD{\Delta} \def\vv{\varepsilon} \def\rr{\rho}
\def\<{\langle} \def\>{\rangle}
  \def\nn{\nabla} \def\pp{\partial} \def\E{\mathbb E}
\def\d{\text{\rm{d}}} \def\bb{\beta} \def\aa{\alpha} \def\D{\scr D}
  \def\si{\sigma} \def\ess{\text{\rm{ess}}}
\def\beg{\begin} \def\beq{\begin{equation}}  \def\F{\scr F}
\def\Ric{\text{\rm{Ric}}} \def\Hess{\text{\rm{Hess}}}
\def\e{\text{\rm{e}}} \def\ua{\underline a} \def\OO{\Omega}  \def\oo{\omega}
 \def\tt{\tilde} \def\Ric{\text{\rm{Ric}}}
\def\cut{\text{\rm{cut}}} \def\P{\mathbb P} \def\ifn{I_n(f^{\bigotimes n})}
\def\C{\scr C}      \def\aaa{\mathbf{r}}     \def\r{r}
\def\gap{\text{\rm{gap}}} \def\prr{\pi_{{\bf m},\varrho}}  \def\r{\mathbf r}
\def\Z{\mathbb Z} \def\vrr{\varrho} \def\ll{\lambda}
\def\L{\scr L}\def\Tt{\tt} \def\TT{\tt}\def\II{\mathbb I}
\def\i{{\rm in}}\def\Sect{{\rm Sect}}  \def\H{\mathbb H}
\def\M{\scr M}\def\Q{\mathbb Q} \def\texto{\text{o}} \def\LL{\Lambda}
\def\Rank{{\rm Rank}} \def\B{\scr B} \def\i{{\rm i}} \def\HR{\hat{\R}^d}
\def\to{\rightarrow}\def\l{\ell}\def\BB{\mathbb B}
\def\8{\infty}\def\I{1}\def\U{\scr U} \def\n{{\mathbf n}}\def\v{V}
\maketitle

\begin{abstract} Let $k\in (d,\infty]$ and consider the $k*$-distance $$\|\mu-\nu\|_{k*}:= \sup\Big\{|\mu(f)-\nu(f)|:\ f\in\B_b(\R^d),\ \|f\|_{\tt L^k}:=\sup_{x\in \R^d}\|1_{B(x,1)}f\|_{L^k}\le 1\Big\}$$
between probability measures on $\R^d$.  The exponential ergodicity  in  $1$-Wasserstein and     $k*$-distances   is derived
for a class of McKean-Vlasov SDEs with small singular interactions measured by $\|\cdot\|_{k*}.$  Moreover, the exponential ergodicity in   $2$-Wasserstein distance and relative entropy is derived when the   interaction term is given by
 $$b^{(0)}(x,\mu) :=\int_{\R^d}h(x-y)\mu(\d y)$$ for some measurable function $h:\R^d\to\R^d$ with small  $\|h\|_{\tt L^k}$.
 \end{abstract} \noindent
 \vskip 2cm

\section{Introduction}

Let $\scr P$ be the set of all probability measures on $\R^d$ equipped with the weak topology. According to \cite{HRW25}, for a constant $k\in [1,\infty]$, let
$$\|\mu\|_{k*}:= \sup\big\{\mu(|f|):\ f\in \B(\R^d),\ \|f\|_{\tt L^k}\le 1\big\},\ \  \mu\in \scr P,$$
where  $\B(\R^d)$ is the set of all measurable functions on $\R^d$, and due to \cite{XXZZ}
$$\|f\|_{\tt L^k}:= \sup_{z\in\R^d} \|1_{B(z,1)}f\|_{L^k}$$
for $B(z,1):=\{x\in \R^d:\ |x-z|\le 1\}$ and $\|\cdot\|_{L^k}$ being the $L^k$-norm with respect to the Lebesgue measure on $\R^d$, where we set   $\|f\|_{\tt L^\infty}=\|f\|_\infty:=\sup_{\R^d}|f|$ when $k=\infty$.  Then
$$\scr P_{k*}:=\big\{\mu\in \scr P:\ \|\mu\|_{k*}<\infty\big\}$$
is complete under the $k*$-distance
$$\|\mu-\nu\|_{k*}:= \sup_{\|f\|_{\tt L^k}\le 1}|\mu(f)-\nu(f)|,\ \ \ \mu,\nu\in \scr P_{k*}.$$

Consider
the following time-homogeneous McKean-Vlasov SDEs on $\R^d$:
\beq\label{E0} \d X_t= b (X_t, \L_{X_t})\d t+  \si(X_t)\d W_t,\ \ t\geq 0,\end{equation}
where  $W_t$ is an $d$-dimensional Brownian motion on a  probability base (i.e. complete filtered  probability space) $(\OO,\{\F_t\}_{t\ge 0},\F,\P)$, $\L_{X_t}$ is the distribution of $X_t$, and
$$ \si: \R^d\to \R^d\otimes\R^d,\ \ \ b:  \R^d\times \scr P_{k*} \to \R^d $$
are measurable.  As explained in \cite{HRW25} that typical examples of $b $ are of  type $b(x,\mu)= b^{(1)}(x)+b^{(0)}(x,\mu)$, where $b^{(1)}$ is a Lipschitz vector field and
\beq\label{0}b^{(0)}(x,\mu):= \int_{\R^d} W(x,y)\mu(\d y),\ \ \ x\in\R^d,\ \mu\in\scr P_{k*}\end{equation}
for a measurable map $W:\R^d\times\R^d\to \R^d$ satisfying
\beq\label{00} |W(x,y)|\le \sum_{i=1}^l \ff {c}{1\land |x-y-z_i|^\bb}\end{equation}
for some constants $l\geq 1,c>0,$ $\bb\in (0,\ff dk)$ and finite many points $\{z_i\}_{1\le i\le l}\subset \R^d.$

We call $X_t$ a solution of \eqref{E0},   if it is an adapted continuous process $X_t$ on $\R^d$   such that
$\L_{X_t}\in \scr P_{k*}$  for all $t> 0$, and $\P$-a.s.
$$X_t=X_0+\int_0^t b(X_s,\L_{X_s})\d s+\int_0^t \si(X_s)\d W_s,\ \ t\ge 0.$$
A couple $(X_t,W_t)$ is called a weak solution of \eqref{E0},   if there is a probability base such that $W_t$ is a $d$-dimensional Brownian motion  and $X_t$ is a strong solution of \eqref{E0}.
If for any $\F_0$-measurable initial value $X_0$ (respectively, any initial distribution $\mu$), \eqref{E0} has a unique strong (respectively, weak) solution, we call the SDE strongly (respectively, weakly) well-posed.
 We call  \eqref{E0} well-posed if it is both strongly and weakly well-posed.
 In this case, denote
$$P_t^*\mu=\L_{X_t}\ \text{for}\ t\ge 0 \  \text{and}\  \L_{X_0}=\mu.$$
A measure $\bar\mu\in   \scr P $ is called an invariant probability measure of $P_t^*$ for \eqref{E0}, if $P_t^*\bar\mu=\bar\mu$ holds for all $t\ge 0$.

Let $\tt{\scr P}$ be a dense subspace of $\scr P$ equipped with a complete metric $\W$.    We call \eqref{E0} (or $P_t^*$) exponentially ergodic in $\W$ or  $\W$-exponentially ergodic, if it has   a unique invariant probability measure $\bar\mu\in \tt{\scr P}$   such that
$$\W(P_t^*\mu,\bar\mu)\le c\e^{-\ll t}\, \W(\mu,\bar\mu),\ \ \ t\ge 0,\ \mu\in  \tt{\scr P}$$
holds for some constants $c,\ll>0$.
A typical choice of  $(\tt{\scr P},\W)$ is  the $q$-Wasserstein space $(\scr P_q,\W_q)$ for $q\in [1,\infty)$, where
\beg{align*} & \scr P_q:=\{\mu\in\scr P,\mu(|\cdot|^q)<\infty\},\\
& \W_q(\gg,\mu)= \inf_{\pi\in \C(\gg,\mu)}\bigg(\int_{\R^d\times\R^d}|x-y|^q\pi(\d x,\d y)\bigg)^{\ff 1 q},\ \ \gg,\mu\in \scr P_q\end{align*}
for   $\C(\gg,\mu)$ being  the set  of all couplings of $\gamma$ and $\mu$.
We also study  the exponential ergodicity in the relative entropy
 $$\Ent(\mu|\gg):= \beg{cases} \mu(\log \ff{\d\mu}{\d\gg}),\ &\text{if}\ \ff{\d\mu}{\d\gg}\ \text{exists},\\
 \infty, \ &\text{otherwise.}\end{cases}$$

When  $b^{(0)}(x,\mu)$ is Lipschitz continuous in $\mu\in\tt{\scr P}$ with respect a nice probability distance $\W$, the exponential  ergodicity of \eqref{E0} have been intensively investigated,
see   \cite{EGZ,LWZ, GLWZ, WFY 2023, Schuh 2024, RW24} for $\W=\W_1$ as well as  \cite{LMJ 2021} with the L\'{e}vy noise, where  \cite{GLWZ} derived
the exponential ergodicity in the mean field entropy by using the uniform log-Sobolev inequality for the associated mean field interacting particle systems;
  see  \cite{RW21} for $\W=\W_2$   where  the exponential ergodicity in $\W_2$ and relative entropy by establishing the  log-Sobolev and log-Harnack inequalities; and see
      \cite{WSB} for $\W$ being the total variation distance.

 However,  the study of  ergodicity is still open when $b(x,\mu)$ contains a singular  term   like in \eqref{0} and \eqref{00}.
In this paper, we prove the exponential ergodicity of \eqref{E0}  with such type singular interactions for $\bb\in (0,1)$,  which include  the  interactions between cytonemes in cellular
communication  \cite{ACPSG}. In this case,   the well-posedness and propagation of chaos have been derived in \cite{HRZ, 21RZ, Zhao}. Recently, when $\bb\in (0,d)$,
 the well-posedness and entropy-cost inequality have been presented in the recent work \cite{HRW25}.

 In Section 2, we prove the exponential ergodicity  in $\|\cdot\|_{k*}$ and $\W_1$ for $b(x,\mu)=b^{(0)}(x,\mu)+b^{(1)}(x),$
 where $b^{(0)}$ is singular in $\mu$ and   $b^{(1)}$ is Lipschitz continuous and dissipative in long distance. When $b^{(0)}$ is given by \eqref{0} for a singular kernel of type $W(x,y)= h(x-y)$, the exponential ergodicity in
 $\W_2$ and $\Ent$ is addressed in  Section 3.

\section{ Exponential Ergodicity in $\W_1$ and $\|\cdot\|_{k*}$}

When $b(x,\mu)=b^{(1)}(x)$ does not depend on $\mu$, the distribution dependent SDE \eqref{E0} reduces to  the classical SDE
\beq\label{E} \d X_t= b^{(1)}(X_t)\d t+\si(X_t)\d W_t,\ \ t\ge 0.\end{equation}
In general, \eqref{E0} can be  regarded as a distribution dependent perturbation of \eqref{E}.
Intuitively, if \eqref{E} is exponential ergodic,   so is  \eqref{E0}  provided the interactions included in
$$b^{(0)}(x,\mu):=b(x,\mu)-b^{(1)}(x)$$
is weak enough in a suitable sense.
In this spirit, we make the following assumptions on
   $b^{(1)}, b^{(0)}$ and $\si.$

\beg{enumerate} \item[{\bf (A)}]   \emph{
  $\si$ is weakly differentiable such that for some constant  $p\in (d,\infty]$
  $$ \|\si\|_\infty+\|\si^{-1}\|_\infty+  \|\nn \si\|_{\tt L^p}<\infty.$$
 Moreover, $b^{(1)}$ is Lipschitz continuous and there exist constants  $K,R\in (0,\infty)$ such that
      $$\<b^{(1)}(x)-b^{(1)}(y),x-y\>\leq -K|x-y|^2,\   \  \text{if }\ |x-y|\geq R.$$}
  \item[{\bf (B)}]  \emph{Let {$k\in (d,\infty]$}. There exists a constant $\eta>0$ such that $b^{(0)}(x,\mu):= b(x,\mu)- b^{(1)}(x)$ satisfies
 $$\|b^{(0)}(\cdot,\mu)\|_\infty\le \eta \|\mu\|_{k*},\ \ \  \|b^{(0)}(\cdot,\mu)-b^{(0)}(\cdot,\nu)\|_\infty\le \eta\|\mu-\nu\|_{k*},\ \ \ \mu,\nu\in \scr P_{k*}.$$}
 \end{enumerate}

It is easy to see that when $\bb\in (0,1)$,     $b^{(0)}$ given by \eqref{0} and \eqref{00} satisfies  {\bf (B)} with $k\in (d,\ff d \bb)$.

 Let $a:=\si\si^*$.  {\bf (A)} implies $\|a\|_\infty+\|a^{-1}\|_\infty+\|\nn a\|_{\tt L^p}<\infty.$ So, by Sobolev's embedding, there exists a constant $c\in (0,\infty)$ such that
 \beq\label{HD} \|a(x)-a(y)\|^2\le c|x-y|^\aa,\ \ \ x,y\in \R^d,\ \aa:= \ff{2(p-d)}p\in (0,2).\end{equation}
 So, according to \cite[Theorem 2.1]{WSB}, where an additional time-spatial singular drift term may be included,
  {\bf (A)} implies the well-posedness and the $\W_1$-exponential ergodicity of the SDE \eqref{E},   see also \cite{Eb} for the case where   $\si=I_d$ being the $d\times d$-identity matrix. When
  {\bf (A)} holds with constant $\si$,    the exponential ergodicity in  $\W_q$ ($q\ge 1$) has been proved in \cite{LW}. Moreover,
   the exponential ergodicity in $\W_2$  is derived in \cite{W2020} for  diffusions on manifolds with negative curvature.

It is easy to see that  {\bf (A)} and {\bf (B)} imply the assumption  $(A_1)$ in \cite{HRW25}   for $p=\infty$, so that by  \cite[Theorem 2.1]{HRW25},  for any initial value $X_0$ (respectively, initial distribution $\mu$), \eqref{E0} has a unique solution, and $P_t^*\mu:=\L_{X_t}$ for $\L_{X_0}=\mu$  satisfies
\beq\label{CPK} \sup_{\mu\in \scr P} \sup_{t\in (0,T]} t^{\ff d{2k}} \|P_t^*\mu \|_{k*}<\infty,\ \ \ T\in (0,\infty),\end{equation}
\begin{equation*} \E\bigg[\sup_{t\in [0,T]} |X_t|^n\bigg]<\infty,\ \ \ \text{if}\ \E|X_0|^n<\infty,\ \ \ T\in (0,\infty),\ n\in\mathbb N.
\end{equation*}

The following result shows that  \eqref{E0} is exponential ergodic in $\W_1$ and $\|\cdot\|_{k*}$  when the interaction is  weak enough in the $k*$-distance  (i.e.  $\eta$ is  small enough).

\begin{thm}\label{T0} Assume {\bf (A)} and {\bf (B)}.
\begin{enumerate}
\item[$(1)$] Any invariant probability measure $\bar\mu$ of $\eqref{E0}$ has density $\bar\rr(x):=\ff{\bar\mu(\d x)}{\d x},$    and
\begin{equation*} \big\|\bar\rr\big\|_\infty+ \bar\mu(\e^{\theta|\cdot|^2})<\infty\ \text{for\ } \theta<K,\end{equation*}
and $ \big\|\bar\rr\big\|_\infty$ is locally bounded in $\eta \in [0,\infty)$.
\item[$(2)$] There  exists a constant $\eta_0>0$ depending only on  $(a,b^{(1)},k,d)$, such that   if  $\eta\le \eta_0$   then $\eqref{E0}$ has a unique invariant probability $\bar{\mu}$
and
\beq\label{EX1}
 \W_1(P_t^\ast\mu,\bar{\mu})\le c \e^{-\ll t}\, \W_1(\mu,\bar\mu),\ \ \ \mu\in \scr P_1,\ t\ge 0,\end{equation}
 \beq\label{EX2}  \|P_t^*\mu-\bar\mu \|_{k*}\le c \e^{-\ll t} (t\wedge1)^{-\ff 1 2-\ff d{2k}} \W_1(\mu,\bar\mu),\ \ \ \mu\in \scr P_1,\ t> 0\end{equation}
hold for some constants $c,\ll\in (0,\infty).$
\end{enumerate}
\end{thm}

The above theorem ensures the exponential ergodicity of  singular McKean-Vlasov SDEs with small enough interaction  $\eta$. In general,   the phase transition may happen for large $\eta$ so that the ergodicity fails.
Below is a typical example of this phenomenon.

\begin{rem}
Let $d=1$, $b^{(1)}(x)=-x,$ $\sigma=\sqrt{2}$, $g(x)=\frac{x}{|x|^{1+\beta}}1_{\{x\neq 0\}}$ for some $\beta\in(0,1)$, $k\in(d,\frac{d}{\beta})$, $\hat{g}(x)=\frac{g}{\|g\|_{\tilde{L}^k}}$, $b^{(0)}(x,\mu)=b^{(0)}(\mu)=\eta \int_{\R} \hat{g}(y)\mu(\d y)$
for some $\eta>0$. Then \eqref{E0} reduces to
\begin{align}\label{kty}\d X_t=-X_t\d t+\eta\int_{\R}\hat{g}(y)\L_{X_t}(\d y)\d t+\sqrt{2}\d W_t.
\end{align}
Since $g$ is odd, it is easy to see that $\mu_0:=\frac{1}{\sqrt{2\pi}}\e^{-\frac{y^2}{2}}\d y$ is an invariant probability measure of \eqref{kty}.
By Theorem \ref{T0}(2), this  SDE is exponential ergodic when $\eta$ is small enough.

On the other hand, we have  $$\zeta(a)=\int_{\R} \hat{g}(y)\frac{1}{\sqrt{2\pi}}\e^{-\frac{|y-a|^2}{2}}\d y>0,\ \ a> 0.$$
If  $ \eta= \ff {a}{\zeta(a)}$ for some $a>0$, then $-a+\eta\zeta(a)=0$ so that letting
 $$\L_{X_t}=\mu_a:=\frac{1}{\sqrt{2\pi}}\e^{-\frac{|y-a|^2}{2}}\d y,$$    the SDE \eqref{kty} becomes
 $$\d X_t=-(X_t-a)\d t  +\sqrt{2}\d W_t,$$
 which has a unique   invariant probability measure  $\mu_a$. Thus,  if $ \eta= \ff {a}{\zeta(a)}$ for some $a>0$, then $\mu_a$ is another invariant probability measure of \eqref{kty}.

Noting that  only  if  $$\eta< \hat\eta  :=\inf_{a>0}\frac{a}{\zeta(a)},$$ the equation $-a+\eta\zeta(a)=0$
 has a unique solution $a=0$. So,    $\hat\eta> \eta_0>0.$

 It is clear that $b^{(0)}(\mu):= \eta \int_{\R} \hat{g}(y)\mu(\d y)$ is not Lipschitz continuous in $\mu$ under the Wasserstein distance, nor under the weighted variation distance, since the kernel
 $\hat{g}$ is not locally bounded.
\end{rem}

To prove the existence and uniqueness of  the invariant probability measure  of \eqref{E0}, we consider
  the following SDE with frozen distribution parameter $\mu\in \scr P_{k*}$:
\beq\label{DC} \d {X}_t^\mu=b^{(1)}({X}_t^\mu)\d t+b^{(0)}({X}_t^\mu,\mu)\d t+\si({X}_t^\mu)\d W_t,\ \ \ t\geq 0.\end{equation}
By {\bf (B)}, we have
\begin{align*}
|b^{(0)}(x,\mu)|\leq \eta\|\mu\|_{k*}<\infty,\ \ x\in\R^d,\ \mu\in \scr P_{k*}.
\end{align*}
According to  \cite[Theorem 2.1]{WSB}, this together with  {\bf (A)} implies that \eqref{DC} is well-posed and $\W_1$-exponential ergodic.  Let
$\Phi(\mu)$ be the unique invariant probability measure of \eqref{DC}, and let
 $$P_t^{\mu*}\nu=\L_{X_t^\mu}\ \text{for}\ t\ge 0 \ \text{and}\ \L_{X_0^\mu}=\nu\in \scr P.$$ Then
\beq\label{INV} P_t^{\mu*} \Phi(\mu)=\Phi(\mu),\ \ \ t\ge 0,\ \ \mu\in \scr P_{k*}.\end{equation}
Let  $P_t^\mu$ be the diffusion semigroup associated   with \eqref{DC}. Then
$$P_t^\mu f(x):=\int_{\R^d} f\d(P_t^{\mu*}\dd_x),\ \ \ t\ge 0,\ f\in \B_b(\R^d),\ x\in\R^d,$$
where $\dd_x$ is the Dirac measure at $x$. For a (real/vector/matrix-valued)  Lipschitz continuous function $f$ on $\R^d$, let
$$\|\nn f\|_\infty:=\sup_{x\ne y} \ff{|f(x)-f(y)|}{|x-y|}$$ denote its Lipschitz constant.
We have the following result  on  both $\Phi(\mu)$ and the $\W_1$-exponential ergodicity of $P_t^{\mu*}.$

 \beg{lem}\label{L1} Assume {\bf (A)} and {\bf (B)}. Then   the following assertions hold.
\begin{enumerate}
\item[$(1)$] $\Phi$ is a map from  $\scr P_{k*}$ to $\scr P_{k*}$, and  there exists a constant $c \in (0,\infty)$ depending only on $(a,b^{(1)},k,d)$ such that for any constant
$$M\ge M(\eta):= c  (1\lor \eta)^{\ff d{k-d}},$$   the set
 $$\scr P_{k*}^M:=\{\mu\in\scr P_{k*},\ \  \|\mu\|_{k*}\leq M\}$$
 is $\Phi$-invariant.
 \item[$(2)$] For any $\theta\in (0,K)$,
 $$\Phi(\mu)\big(\e^{\theta|\cdot|^2}\big)<\infty,\ \  \ \mu\in \scr P_{k*}.$$
 Consequently, $\Phi: \scr P_{k*}\to\scr P_{k*}\cap \scr P_q$ for any $q\in [1,\infty)$.
\item[$(3)$] There exists a constant
$ \ll(\eta)\in (0,\infty)$ decreasing in $ \eta>0$ such that
\begin{align*}
\W_1({P}_t^{\mu *}  \mu,P_t^{\mu*}\tt\mu)\leq \ff 1 {\ll(\eta)}  \e^{-\lambda(\eta) t}\W_1(\mu,\tt\mu),\ \ t\ge 0,\  \mu,\tt\mu\in \scr P_{1}, \mu\in \scr P_{k*}^{M(\eta)}.
\end{align*}
 Consequently, for any Lipschitz continuous function $f$,
 $$\|\nn P_t^\mu f\|_\infty\le \ff 1 {\ll(\eta)}  \e^{-\lambda(\eta) t}\|\nn f\|_\infty,\ \ \  t\ge 0,\ \mu\in \scr P_{k*}^{M(\eta)}.$$
\end{enumerate}
 \end{lem}

 \beg{proof}   Let $\mu\in \scr P_{k*}^M$.
  By Duhamel's formula,  see \cite[(5.33)]{HRW25}, we have
\beq\label{DH} P_t^\mu f=\hat P_t f+\int_0^t P_s^\mu\<b^{(0)}(\cdot,\mu),\nn \hat P_{t-s} f\>\d s,\ \ f\in \B_b(\R^d),\end{equation}
where  $\hat P_t$ is  the diffusion semigroup generated by
\beq\label{HTL} \hat L:= \ff 1 2 {\rm tr}\{a\nn^2\} + b^{(1)}\cdot\nn.\end{equation}
   By \cite[Lemma 5.3]{HRW25}, there exists a constant $c_1\in (1,\infty)$ depending only on $(a,b^{(1)})$ such that
 \beq\label{YS} \|\nn^{i} \hat P_t\|_{\tt L^{p_1}\to\tt L^{p_2}}\le c_1 t^{-\ff {d(p_2-p_1)}{2p_1p_2}-\ff i 2},\ \ t\in (0,1],\ i=0,1,\ 1\le p_1\le p_2\le \infty.\end{equation}
Below we complete the proof in four steps.

 (a)  By  \eqref{INV}, \eqref{DH}, \eqref{YS}, $\|P_t^\mu\|_{\tt L^\infty\to\tt L^\infty}=1$  and {\bf (B)}, we obtain
 \beg{align*}&\|\Phi(\mu)\|_{k*}= \sup_{\|f\|_{\tt L^k}\le 1}  |\Phi(\mu)(f)| = \sup_{\|f\|_{\tt L^k}\le 1}   |\Phi(\mu)(P_t^{\mu} f)|\\
& \le \|\hat P_t\|_{\tt L^k\to\tt L^\infty} +\int_0^t \eta \|\mu\|_{k*} \|\nn\hat P_{t-s}\|_{\tt L^k\to \tt L^\infty}\d s\\
 &\le c_1  t^{-\ff d{2k}} + c_1\eta M \int_0^t (t-s)^{-\ff 1 2 -\ff d {2k}}\d s\\
 &=  c_1  t^{-\ff d{2k}} +   \ff {2kc_1\eta M }{k-d} t^{ \ff {k-d}{2k} },\ \ t\in (0,1].\end{align*}
 This implies $\Phi(\mu)\in \scr P_{k*}.$
 By taking
 \beq\label{T}  t= t(\eta):=\Big(\ff{k-d}{4 k c_1 (1\lor\eta)}\Big)^{\ff{2k}{k-d}} \in (0,1),\end{equation}
 we derive
 $$\|\Phi(\mu)\|_{k*}\le c_1 t(\eta)^{-\ff d{2k}} + \ff 1 2 M\le M,$$
 provided
 $$M\ge M(\eta):= 2 c_1 t(\eta)^{-\ff d {2k}} = c  (1\lor\eta)^{-\ff d{k-d}},$$
 where
 $$c := 2c_1 \Big(\ff{k-d}{4 k c_1}\Big)^{-\ff d {k-d}} \in (0,\infty)$$
 depends only on $(a,b^{(1)},k,d)$.  So, for  any $M\ge M(\eta)$, we have $\Phi: \scr P_{k*}^M\to \scr P_{k*}^M.$

 (b) By {\bf (A)} and  {\bf (B)},   for any $\theta\in (0,   K) $ we find  constants $C_1,C_2\in (0,\infty)$ depending only on $(a,b^{(1)},\eta,K,\|\mu\|_{k*},\theta)$ such that
$$\d\e^{\theta |X_t^\mu|^2} \le \big(C_1- C_2 |X_t^\mu|^2\e^{\theta |X_t^\mu|^2} \big)\d t + \d M_t,\ \ t\ge 0 $$
holds for some local martingale $M_t$.
Since $\Phi(\mu)$ is the invariant probability measure of $X_t^\mu$,    this implies
$$\Phi(\mu)\big(|\cdot|^2 \e^{\theta|\cdot|^2}\big)\le \ff {C_1}{C_2}<\infty.$$
So, the second assertion holds.

(c)   By {\bf (B)}, we have
\begin{align}\label{bob}\|b^{(0)}(\cdot,\mu)\|_{\infty}\leq \eta M(\eta)=:c_0(\eta),\ \ \mu\in \scr P_{k*}^{M(\eta)}.
\end{align}
Combining this with {\bf (A)} and {\bf (B)}, we find a constant $c_2(\eta)\in (0,\infty)$ increasing in $\eta$ such that
 \beq\label{I01}  \beg{split}&
 \<b^{(1)}(x)-b^{(1)}(y)+ b^{(0)}(x,\mu)-b^{(0)}(y,\mu), x-y\> \\
 &\le c_2(\eta)|x-y| - K|x-y|^2,\ \
  \ \ \mu\in \scr P_{k*}^{M(\eta)},\ x, y\in\R^d.
\end{split} \end{equation}
On the other hand, by {\bf (A)} and \eqref{HD}, we find  constants $c_3,\bb>0$ and $\aa\in (0,2)$ such that
$a\ge 2 \bb I_d$, where $I_d$ is the $d\times d$-identity matrix, and $\hat\si:=\ss{a-\bb I_d}$ satisfies
\beq\label{I02} \|\hat \si(x)-\hat\si(y)\|_{HS}^2\le c_3 |x-y|^{\aa},\ \ \ x,y\in \R^d.\end{equation}
We now make use the following coupling with decomposition of noise introduced in \cite{PR}: for two independent $d$-dimensional Brownian  motions $B_t^1$ and $B_t^2$,
\beg{align*}  & \d  {X}_t   =\big(b^{(1)}(X_t) +b^{(0)}({X}_t, \mu)\big) \d t+ \ss{\bb}  \d B_t^1+\hat\si(X_t)  \d B_t^2,\\
& \d Y_t   = \big(b^{(1)}(Y_t) +b^{(0)}({Y}_t, \mu)\big) \d t
  +\ss{\bb} \big\{u(X_t,Y_t)1_{\{t<\tau\}}+ I_d1_{\{t\ge \tau\}}\big\}   \d B_t^1+\hat{\sigma}(Y_t) \d B_t^2,
\end{align*}
where  $X_0,Y_0$ satisfy $\L_{X_0}=\tt\mu$, $\L_{Y_0}=\mu$ and
\beq\label{I03} \E[|X_0-Y_0|]=\W_1(\mu, \tt\mu),\end{equation}
 $\tau:= \inf\{t\ge 0: X_t=Y_t\}$ is the coupling time, and
 $$u(x,y):=I_d- 2\ff{(x-y)\otimes (x-y)}{|x-y|^2},\ \ x\neq y$$
 is the mirror reflection operator.

Since $\si\si^*=a=\bb I_d+\hat\si^2,$ we find two independent $d$-dimensional Brownian motions $W_t^1$ and $W_t^2$ such that
\beg{align*}&\ss{\bb}  \d B_t^1+\hat\si(X_t)  \d B_t^2= \si(X_t)\d W_t^1,\\
&\ss{\bb} \big\{u(X_t,Y_t)1_{\{t<\tau\}}+ I_d1_{\{t\ge \tau\}}\big\}    \d B_t^1+\hat{\sigma}(Y_t) \d B_t^2=\si(Y_t)\d W_t^2.\end{align*}
By the weak uniqueness of \eqref{DC}, we have
$$\L_{X_t}= P_t^{\mu*}\tt\mu,\ \ \ \L_{Y_t}=P_t^{\mu*}\mu,$$
so that
\beq\label{I03'}\W_1(P_t^{\mu*}\mu, P_t^{\mu*}\tt\mu)\le \E[|X_t-Y_t|],\ \ \ t\ge 0.\end{equation}
Hence, to prove  assertion (3), it suffices to estimate $\E[|X_t-Y_t|].$

(d)  By \eqref{I01}, \eqref{I02} and the It\^o-Tanaka formula, we obtain
\beq\label{I04} \d |X_t-Y_t|\le \phi( |X_t-Y_t|)\d t + \d M_t,\ \ \ t\in [0,\tau),\end{equation}
where  $\phi(r):= c_2(\eta)+ c_3r^{-(1-\aa)^+} -K r,\  r>0$ and
$$\d M_t:= \Big\<\ff{X_t-Y_t}{|X_t-Y_t|},\ 2\ss\bb \d B_t^1+ \big(\hat\si(X_t)-\hat\si(Y_t)\big)\d B_t^2\Big\>$$
satisfies
\beq\label{I05} \d\<M\>_t\ge 4\bb\d t,\ \ \ t\in [0,\tau).\end{equation}
Define
$$\psi(r):=\int_0^r \e^{-\ff 1{2 \bb} \int_0^u \phi(v)\d v} \d u \int_u^\infty s \e^{\ff 1{2 \bb} \int_0^s\phi(v)\d v}\d s,\ \ \ r\ge 0.$$
We claim that
\beq\label{I06} \psi(0)=0, \ \ \psi>0,\ \ 2 \bb   \psi'' + \phi(r) \psi'(r) =-r,\ \ \psi''\le 0,\ \ \end{equation} so that by \eqref{I04}, \eqref{I05} and It\^o's formula,
\beq\label{I07} \d \psi(|X_t-Y_t|)\le - |X_t-Y_t|\d t + \psi'(|X_t-Y_t|)\d M_t,\ \ \ t\in [0,\tau).\end{equation}
The first three properties in \eqref{I06} are trivial, it suffices to verify $\psi''\le 0$. Noting  that $\phi$ is decreasing with
$$\lim_{r\to 0} \phi(r) >0,\ \ \ \lim_{r\to\infty} \phi(r)=-\infty,$$
we find  a unique constant $r_0\in (0,\infty)$ such that $\phi(r_0)=0,\ \phi|_{(0,r_0)}>0$ and $\phi|_{(r_0,\infty)}<0.$
Noting that
\beq\label{I08} \beg{split}&\psi'(r)= \e^{-\ff 1{2 \bb} \int_0^r \phi(v)\d v} \d u \int_u^\infty s \e^{\ff 1{2 \bb} \int_0^s\phi(v)\d v}\d s,\\
&\psi''(r)= -\ff 1{2 \bb} \phi(r)\psi'(r) -r,\ \ \ r>0,\end{split}\end{equation}
it is clear that $\psi''(r)\le 0$ for $r\le r_0,$   since $\psi'\ge 0$ and in this case $\phi(r)\ge 0.$ On the other hand, since
$$\ff{r}{-\phi(r)}=\ff 1 {K- c_2 r^{-1}- c_3 r^{\aa-2}},\ \ \ r>r_0$$
is decreasing, we have
\begin{align*}
&\int_r^\infty s\e^{\frac{1}{2\beta}\int_0^s\phi(v)\d v}\d s
 =\int_r^\infty \frac{2\beta s}{ \phi(s)}\left(\frac{\d}{\d s} \e^{\frac{1}{2\beta}\int_0^s\phi(v)\d v}\right)\d s\\
&= \int_r^\infty 2\bb   \e^{\frac{1}{2\beta}\int_0^s\phi(v)\d v} \ff{\d}{\d s}\Big(\ff s{-\phi(s)}\Big)\,\d s\le   -\frac{2\beta r}{ \phi(r)}\e^{\frac{1}{2\beta}\int_0^r\phi(v)\d v},\ \ r> r_0.
\end{align*}
Combining this with \eqref{I08} we conclude that $\psi''(r)<0$ for $r>r_0.$ Thus, \eqref{I06} and \eqref{I07} hold, where \eqref{I06} implies
$\psi(r)\le \psi'(0) r$ and
$$\inf_{r>0} \frac{\psi(r)}{r}=\lim_{r\to\infty}\frac{\psi(r)}{r}=\lim_{r\to \infty}\psi'(r)=\lim_{r\to \infty}\frac{\int_r^\infty s\e^{\frac{1}{2\beta}\int_0^s\phi(v)\d v}\d s}{\e^{\frac{1}{2\beta}\int_0^r\phi(v)\d v}}=\frac{2\beta}{  K},
$$
so that
\begin{align}\label{cop}\frac{2\beta}{ K } r\leq \psi(r)\leq \psi'(0)r,\ \ r\geq 0,
\end{align}
which together with  \eqref{I07} implies
$$\d\psi(|X_t-Y_t|) \le - \ff{1}{\psi'(0)}\psi(|X_t-Y_t|)+ \psi'(|X_t-Y_t|)\d M_t,\ \ \ t\in [0,\tau).$$
Noting that $X_t=Y_t$ for $t\ge \tau$, this ensures
$$\E[\psi(|X_t-Y_t|)]\le \e^{-\ff t {\psi'(0)}} \E[\psi(|X_0-Y_0|)],\ \ \ t\ge 0.$$
Combining this with \eqref{I03}, \eqref{I03'} and \eqref{cop}, and noting that
$$\psi'(0)= \int_0^\infty s \e^{\ff 1{2 \bb} \int_0^s\phi(v)\d v}\d s $$
is increasing in $\eta$ since so is $c_2(\eta)$  in the definition of $\phi$, we finish the proof.
 \end{proof}

Next,  we have the following result on  the invariant probability measure of $P_t^*.$

 \beg{lem}\label{L0} Assume {\bf (A)} and {\bf (B)}. Then there exists a constant $c \in (0,\infty)$ only depending on $(a,b^{(1)})$ such that any invariant probability measure
 $\bar\mu$ of $P_t^*$ satisfies
\beq\label{BM}  \|\bar\mu\|_{k*}\le M(\eta):=c(1\lor \eta)^{\ff d{k-d}},\  \    \  \bar\mu(\e^{\theta |\cdot|^2})<\infty \ \text{for}\    \theta<  K,\end{equation}

 \end{lem}

 \beg{proof}  Since $P_t^*\bar\mu =\bar\mu,$ by \eqref{CPK} we obtain $\|\bar\mu\|_{k*}<\infty$.
 By \eqref{DH} for $\mu=\bar\mu$ and  combining with  $\bar\mu(P_t^{\bar \mu} f)=\bar\mu(f)$, {\bf (B)}  and \eqref{YS}, we obtain
\beg{align*}&\|\bar\mu\|_{k*}= \sup_{\|f\|_{\tt L^k}\le 1}  |\bar\mu(f)| = \sup_{\|f\|_{\tt L^k}\le 1}   |\bar\mu(P_t^{\bar\mu} f)|\\
& \le \|\hat P_t\|_{\tt L^k\to\tt L^\infty} +\int_0^t \eta \|\bar\mu\|_{k*} \|\nn\hat P_{t-s}\|_{\tt L^k\to \tt L^\infty}\d s\\
 &\le c_1  t^{-\ff d{2k}} + c_1\eta\|\bar\mu\|_{k*} \int_0^t (t-s)^{-\ff 1 2 -\ff d {2k}}\d s\\
 &=  c_1  t^{-\ff d{2k}} +    \ff {2kc_1\eta}{k-d} t^{ \ff 1 2- \ff d {2k}} \|\bar\mu\|_{k*},\ \ t\in (0,1].\end{align*}
Taking  $t=t(\eta)$  in \eqref{T}, we derive
$$\|\bar\mu\|_{k*}\le 2 c_1  t(\eta)^{-\ff d{2k}}= M(\eta).$$
This together with Lemma \ref{L1}(2) finishes the proof.
 \end{proof}

Consider  the decoupled SDE for $\mu\in \scr P$ and $s\ge 0$:
\begin{equation*} \d \bar X_{s,t}^\mu= b(\bar X_{s,t}^\mu, P_t^*\mu)\d t +\si(\bar X_{s,t}^\mu) \d W_t,\ \ t\ge s.
\end{equation*}
Under {\bf (A)} and {\bf (B)} this SDE is well-posed. For any $\nu\in \scr P$, let
$$\bar P_{s,t}^{\mu *}\nu:=\L_{\bar X_{s,t}^\mu}\ \text{for}\ t\ge s\ \text{and}\   \L_{\bar X_{s,s}^\mu}=\nu.$$
Let $\bar X_{s,t}^{\mu,x}$ be the unique solution to this SDE with $\bar X_{s,s}^{\mu,x}=x$, the associated diffusion semigroup is defined as
$$\bar P_{s,t}^\mu f(x)= \E[f(\bar X_{s,t}^{\mu,x})],\ \ t\ge s,\ f\in \B_b(\R^d).$$
 Simply denote $\bar P_{t}^\mu=\bar P_{0,t}^\mu$ and $\bar P_{t}^{\mu *}=\bar P_{0,t}^{\mu *}$.
Then
\beq\label{LX} P_t^*\mu= \bar P_t^{\mu*}\mu,\ \ \ \bar\mu= P_t^*\bar\mu = \bar P_t^{\bar\mu*}\bar\mu= P_t^{\bar\mu*}\bar\mu.\end{equation}
By Duhamel's formula, see \cite[(5.33)]{HRW25}, we have
 \begin{equation*} \bar P_{s,t}^\mu f =\hat P_{t-s} f+ \int_s^t \bar P_{s,r}^\mu \<b^{(0)}(\cdot, P_r^*\mu), \nn\hat P_{t-s}f\> \d r,\ \ t\ge s\ge 0,\ f\in\B_b(\R^d).
 \end{equation*}

 \beg{proof}[Proof of Theorem \ref{T0}]
   By Lemma \ref{L0},  we have $\bar\mu(\e^{\theta|\cdot|^2})<\infty$ for $\theta<K.$ So, for the first assertion  it suffices to show
 \beq\label{KO} \sup_{\|f\|_{\tt L^1}\le 1} \bar\mu(|f|)\le H(\eta) \end{equation}
 for some increasing function $H: [0,\infty)\to (0,\infty)$,
 which is equivalent to $\|\bar\rr\|_\infty\le H(\eta)$. In the following, we prove \eqref{KO} and   the second assertion respectively.

(a)  By \cite[Lemma 3.3]{HRW25}, {\bf (A)}, {\bf (B)} and \eqref{BM} imply that for any $p>1$ there exists a constant $c(p,\eta)\in (0,\infty)$ increasing in $\eta$ such that
 $$\|P_t^{\bar\mu}\|_{\tt L^p\to\tt L^\infty}\le c(p,\eta) t^{-\ff{d}{2p}},\ \ \|P_t^{\bar\mu}\|_{\tt L^p\to\tt L^p}\le c(p,\eta),\ \  \  t\in (0,1].$$
  Choosing $p>1$ such that $\ff{d(p-1)}{p}<1$, by combining this with {\bf (B)}, \eqref{DH} and  \eqref{YS},    we obtain
  \beg{align*}&\|P_t^{\bar\mu}\|_{\tt L^1\to \tt L^p}\le \|\hat P_t\|_{\tt L^1\to\tt L^p}+\|b^{(0)}(\cdot,\bar\mu)\|_\infty \int_0^t  \|P_s^{\bar\mu}\|_{\tt L^p\to\tt L^p} \|\nn \hat P_{t-s}\|_{\tt L^1\to\tt L^p}\d s\\
 &\le c_1 t^{-\ff{d(p-1)}{2p}} +  \eta\|\bar\mu\|_{k*}  c(p,\eta)c_1  \int_0^t   (t-s)^{-\ff 1 2-\ff {d(p-1)}{2p}} \d s<\infty,\ \ \ t\in (0,1].\end{align*}
 Noting that $P_2^{\bar\mu*}\bar\mu=\bar\mu$, we drive
 $$\sup_{\|f\|_{\tt L^1}\le 1} \bar\mu(|f|) = \sup_{\|f\|_{\tt L^1}\le 1} \bar\mu(P_2^{\bar\mu} |f|)\le \|P_{1}^{\bar\mu} P_{1}^{\bar\mu} \|_{\tt L^1\to\tt L^\infty} \le \|P_{1}^{\bar\mu} \|_{\tt L^p\to\tt L^\infty} \|P_{1}^{\bar\mu} \|_{\tt L^1\to\tt L^p}.$$
  So, \eqref{KO} holds for {$H(\eta):=c(p,\eta)\left(c_1  +  \eta\|\bar\mu\|_{k*}  c(p,\eta)c_1  \int_0^1   s^{-\ff 1 2-\ff {d(p-1)}{2p}} \d s\right)$.}

 It remains to find a constant $\eta_0>0$ depending only on $(a,b^{(1)}, k,d)$ such that the second assertion hold.
 All constants below depend only on $(a,b^{(1)},k,d,\eta)$, but we denote them by $c_i(\eta)$ for $i\ge 0$  to emphasize  the dependence on $\eta$ such that   $\eta_0$ can be determined accordingly.

(b) Existence and uniqueness of $\bar\mu$. By Lemma \ref{L1}, we have
 $$\Phi:\ \scr P_{k*}^{M(\eta)}\cap \scr P_1\to \scr P_{k*}^{M(\eta)}\cap \scr P_1.$$
It suffices to show that $\Phi$ is contractive on  $\scr P_{k*}^{M(\eta)}\cap \scr P_1$ under the complete metric
$$\W(\mu,\nu):= \|\mu-\nu\|_{k*}+\W_1(\mu,\nu).$$
 By \cite[(2.16)]{HRW25}, there exists a constant $c_0(\eta)\in [1,\infty)$ increasing in $\eta$ such that
\begin{align}\label{kgy}\|P_1^{\mu*}\mu- P_1^{\mu*}\tt\mu\|_{k*}\le c_0(\eta) \W_1(\mu,\tt\mu),\ \ \mu\in \scr P_{k*}^{M(\eta)}.
\end{align}
Since $\ll(\eta)$ is decreasing in $\eta$, there exists a unique  $T_\eta\in [1,\infty)$ which is increasing in $\eta$ such that
\begin{equation*}  \ff {1+c_0(\eta)} {\ll(\eta)}  \e^{-\lambda(\eta) (T_\eta-1)} =\ff 1 2,
\end{equation*}   so that Lemma \ref{L1}(3) and \eqref{kgy} imply
\begin{equation*}  \|P_{T_\eta}^{\mu*}\mu- P_{T_\eta}^{\mu*}\tt\mu\|_{k*} + \W_1(P_{T_\eta}^{\mu*}\mu, P_{T_\eta}^{\mu*}\tt\mu ) \le \ff 1 2 \W_1(\mu,\tt\mu),\ \ \mu,\tt\mu\in \scr P_{k*}^{M(\eta)}\cap \scr P_1.
\end{equation*}
Noting that $\Phi(\mu)\in  \scr P_{k*}^{M(\eta)}\cap \scr P_1$ for $\mu\in \scr P_{k*}^{M(\eta)},$ we obtain
\beq\label{13} \beg{split} &\W(\Phi(\mu),\Phi(\nu))= \W(P_{T_\eta}^{\mu*}\Phi(\mu), P_{T_\eta}^{\nu*}\Phi(\nu))\\
&\le  \W(P_{T_\eta}^{\mu*}\Phi(\mu), P_{T_\eta}^{\mu*}\Phi(\nu))+  \W(P_{T_\eta}^{\mu*}\Phi(\nu), P_{T_\eta}^{\nu*}\Phi(\nu))\\
&\le \ff 1 2\W_1(\Phi(\mu),\Phi(\nu))+ \W(P_{T_\eta}^{\mu*}\Phi(\nu), P_{T_\eta}^{\nu*}\Phi(\nu)),\ \ \ \mu,\nu\in \scr P_{k*}^{M(\eta)}.\end{split}\end{equation}
According to \cite[Lemma 5.3]{HRW25},  from   {\bf (A)} and \eqref{bob}  we find a constant $c_1(\eta)\in (0,\infty)$ increasing in $\eta$ such that
 \beq\label{ES3} \beg{split}&\|\nabla^i {P}_{t}^\mu\|_{\tt L^{p_1}\to\tt L^{p_2}}\le c_1(\eta)    t^{-\frac{i}{2}-\ff{d(p_2-p_1)}{2p_1p_2}},\\
&\quad \ \ t\in (0,T_\eta],\ \ \mu\in\scr P_{k*}^{M(\eta)},\ \ i=0,1,\ 1\le p_1\le p_2\le\infty. \end{split} \end{equation}

 By {\bf (B)}, \eqref{ES3}  and the Duhamel formula  \cite[(5.33)]{HRW25}, we obtain
 \beq\label{14} \beg{split}  & \|P_{T_\eta}^{\mu*} \Phi(\nu)- P_{T_\eta}^{\nu*} \Phi(\nu)\|_{k*}=\sup_{\|f\|_{\tt L^k}\le 1} \big|\Phi(\nu)(P_{T_\eta}^\mu f- P_{T_\eta}^\nu f)\big| \\
 &\le\sup_{\|f\|_{\tt L^k}\le 1}\int_0^{T_\eta} \big\|P_s^\mu\<b^{(0)}(\cdot,\mu)- b^{(0)}(\cdot,\nu),\ \nn P_{t-s}^\nu f\>\big\|_\infty\d s\\
 &\le \eta \|\mu-\nu\|_{k*} \int_0^{T_\eta} c_1(\eta) t^{-\ff 1 2-\ff d{2k}}\d t= \eta \ff{2kc_1(\eta)}{k-d}T_\eta^{\ff{k-d}{2k}}\|\mu-\nu\|_{k*}.\end{split}\end{equation}
 On the other hand, by Lemma \ref{L1}(3) and Kantonovich dual formula, we derive
\beg{align*}\nonumber & \W_1(P_{T_\eta}^{\mu*} \Phi(\nu), P_{T_\eta}^{\nu*} \Phi(\nu)) =\sup_{\|\nn f\|_\infty\le 1} \big|\Phi(\nu)(P_{T_\eta}^\mu f- P_{T_\eta}^\nu f)\big| \\
 &\le \sup_{\|\nn f\|_\infty\le 1}\int_0^{T_\eta} \big\|P_s^\mu\<b^{(0)}(\cdot,\mu)- b^{(0)}(\cdot,\nu),\ \nn P_{t-s}^\nu f\>\big\|_\infty\d s\\
\nonumber & {\le  \eta    \|\mu-\nu\|_{k*}\int_0^{T_\eta}\ff 1 {\ll(\eta)}  \e^{-\lambda(\eta) t}\d t\leq \eta\frac{T_\eta}{\lambda(\eta)}    \|\mu-\nu\|_{k*}}.
\end{align*}
Combining this with \eqref{14}, we derive
$$\W(P_{T_\eta}^{\mu*} \Phi(\nu), P_{T_\eta}^{\nu*} \Phi(\nu))\le \eta \max\Big\{\ff{2kc_1(\eta)}{k-d}T_\eta^{\ff{k-d}{2k}}, {\frac{T_\eta}{\lambda(\eta)} } \Big\}\W(\mu,\nu).$$
Since the upper bound is increasing in $\eta$ and goes to $0$ as $\eta\to 0$,  we find a constant $\eta_0>0$ depending only on $(a,b^{(1)},k,d)$ such that
 $$\W(P_{T_\eta}^{\mu*} \Phi(\nu), P_{T_\eta}^{\nu*} \Phi(\nu))\le \ff 1 4 \W(\mu,\nu),\ \ \eta\in (0,\eta_0],\ \mu,\nu\in \scr P_{k*}^{M(\eta)}.$$
 This together with \eqref{13} implies that $\Phi$ is $\W$-contractive when $\eta\le \eta_0.$

 (c) $\W_1$-exponential ergodicity.  By Lemma \ref{L1}(3), it holds   \beq\label{EN1} \beg{split}&\W_1(P_t^*\mu,\bar\mu)=  \W_1(\bar P_t^{\mu*}\mu,P_t^{\bar\mu*}\bar\mu)\\
 &\le  \W_1(\bar P_t^{\mu*}\mu, P_t^{\bar\mu*}\mu)+  \W_1(P_t^{\bar\mu*}\mu, P_t^{\bar\mu*}\bar\mu)\\
 &\le \W_1(\bar P_t^{\mu*}\mu, P_t^{\bar\mu*}\mu)+ \ff 1 {\ll(\eta)}  {\e^{-\ll(\eta)t}\W_1(\mu,\bar\mu),\ \ t\ge 0},\ \ \mu\in\scr P_1,\end{split}\end{equation}
 where $\ll(\eta)$ is decreasing in $\eta$. Letting   $\tilde{T}_\eta\in (0,\infty)$ be  increasing in $\eta$  such that
 $$\ff 1 {\ll(\eta)} \e^{-\ll(\eta)\tilde{T}_\eta}=\ff 1 4,$$
 we derive
\beq\label{EN2} \W_1(P_t^*\mu,\bar\mu)
 \le  \W_1(\bar P_t^{\mu*}\mu, P_t^{\bar\mu*}\mu)  + {\ff 1 4\W_1(\mu,\bar\mu),\ \ t\ge \tilde{T}_\eta}. \end{equation}
 By Lemma \ref{L1}, we have
 $$\|\nn P_t^{\bar\mu*} f\|_\infty\le \ff 1 {\ll(\eta)}\e^{-\ll(\eta)t}  \|\nn f\|_\infty,\ \ t\ge 0.$$
 Combining this with  {\bf (B)} and  {the Duhamel formula}, we find a constant $\kk_1(\eta)\in (0,\infty)$ increasing in $\eta$ such that
 \beg{align*} & \W_1(\bar P_t^{\mu*} \mu, P_t^{\bar \mu*} \mu)= \sup_{\|\nn f\|_\infty\le 1} |\mu(\bar P_t^\mu f-P_t^{\bar \mu}f)| \\
& { \le \sup_{\|\nn f\|_\infty\le 1}\int_0^{t} \big\|\bar{P}_s^\mu\<b^{(0)}(\cdot,P_s^*\mu)- b^{(0)}(\cdot,\bar{\mu}),\ \nn P_{t-s}^{\bar{\mu}} f\>\big\|_\infty\d s}\\
 &\le \ff 1 {\ll(\eta)} \int_0^t \|b^{(0)}(\cdot,P_s^*\mu)- b^{(0)}(\cdot,\bar\mu)\|_\infty \d s\le  \ff \eta {\ll(\eta)} \int_0^t \|P_s^*\mu-\bar\mu\|_{k*} \d s,\ \ t\geq 0.\end{align*}
 On the other hand, under  {\bf (A)} and {\bf (B)}, the estimate \cite[(2.16)]{HRW25} holds for  $T=(2\tilde{T}_\eta)\vee1$, which implies
\beq\label{PLN} \|P_{t}^*\mu-\bar\mu\|_{k*}= \|P_{t}^*\mu-P_t^{*} \bar\mu\|_{k*}\le \kk_1(\eta) t^{-\ff 1 2-\ff d{2k}}\W_1(\mu,\bar\mu),\ \ \  t\in (0, (2\tilde{T}_\eta)\vee 1]\end{equation}
for some  constant $\kk_1(\eta)\in (0,\infty)$  increasing in $\eta$.
 Thus, we find a constant $\kk_2(\eta)\in (0,\infty)$ which is increasing in $\eta$ such that
 \beq\label{N3} \W_1(\bar P_t^{\mu*} \mu, P_t^{\bar \mu*} \mu) \le  \eta \kk_2(\eta) \W_1(\mu,\bar\mu),\ \ t\in [0,2 \tilde{T}_\eta].\end{equation}
 Combining \eqref{N3} with \eqref{EN1} we derive
 $$\sup_{t\in [0, 2\tilde{T}_\eta]} \W_1(P_t^*\mu,\bar\mu) \le \left( \eta \kk_2(\eta) +\ff 1 {\ll(\eta)}\right) \W_1(\mu,\bar\mu).$$
 Moreover,   by taking   $\eta_0>0$ such that  {$\eta_0 \kk_2(\eta_0)\le \ff 1 4$},  we deduce from \eqref{N3}  and  \eqref{EN2} that $\eta\le \eta_0$ implies
 $$\W_1(P_t^*\mu,\bar\mu)\le \ff 1 2 {\W_1(\mu,\bar\mu),\ \ \ t\in [\tilde{T}_\eta, 2\tilde{T}_{\eta}]}.$$
 Therefore, there exist constants $c,\ll>0$ such that \eqref{EX1} holds.
By combining \eqref{EX1}  with \eqref{PLN} and the semigroup property of $P_t^*$, we derive
 \beg{align*}&\|P_t^*\gg-\bar\mu\|_{k*}=\|P_{1}^*P_{t-1}^*\gg-\bar\mu\|_{k*}\\
 &\le \kk_1(\eta)     \W_1(P_{t-1}^*\gg,\bar\mu) \le \kk_1(\eta)   c\e^{-\ll (t-1)} \W_1(\gg,\bar\mu),\ \ \ t>1,\end{align*}
which together with \eqref{PLN} for $t\le 1$ implies    \eqref{EX2}  for the same $\ll>0$ but a  different constant  $c >0$.
   \end{proof}

\section{Exponential ergodicity in $\W_2$ and $\Ent$}

To derive the exponential ergodicity in $2$-Wasserstein distance, we consider
\beq\label{BN} b^{(0)}(x,\mu)=\int_{\R^d}h(x-y)\mu(\d y),\ \ \mu\in\scr P_{k*}\end{equation}  for some $k\in(d,\infty]$ and
  $h:\R^d\to\R^d$ with $\|h\|_{ \tilde{L}^{k}}<\infty.$

We make the following assumption.

\beg{enumerate}
\item[{\bf(C)}]  \emph{ Let $b(x,\mu) = b^{(0)}(x,\mu)+ b^{(1)}(x).$   }
\item[$(C_1)$] \emph{$\si$ is differentiable,    $\nn \si$ is  H\"older continuous, and
 $$\|\si\|_\infty+\|\si^{-1}\|_\infty+\|\nn \si\|_\infty+\|\nn b^{(1)}\|_\infty<\infty.$$}
 \item[$(C_2)$]  \emph{There exist constants $K\in (0,\infty)$ and $R\in [0,\infty)$ such that
$$\ff 1 2 \|\si(x)-\si(y)\|_{HS}^2 + \<b^{(1)}(x)-b^{(1)}(y),x-y\>\leq -K|x-y|^2\ \text{if}\  |x-y|\geq R.$$}
\item[$(C_3)$]
   \emph{ $b^{(0)}$ is given by \eqref{BN} with $\eta:=\|h\|_{ \tilde{L}^{k}}<\infty$ for some $k\in (d,\infty].$   }\end{enumerate}

Obviously, {\bf (C)} implies   {\bf (A)} and {\bf (B)} with the same constant $\eta$, so that  Theorem \ref{T0} applies.

 \begin{thm}\label{T0z} Assume {\bf (C)} and let $\eta_0,\ll>0$ be in Theorem $\ref{T0}$.
Then the following assertions hold.
\begin{enumerate}
\item[$(1)$] If $R=0$,  then there exists a constant $\bar{\eta}_0\in (0,\eta_0]$ depending only on $(a,b^{(1)},k,d)$ such that when $\eta\in(0,\bar{\eta}_0]$,
\begin{align*}\W_2(P_t^\ast\mu,\bar{\mu})^2+ \mathrm{Ent}(P_t^\ast\mu|\bar{\mu})\leq c\e^{-2\lambda t}\W_2(\mu,\bar{\mu})^2,\ \ \mu\in\scr P_2,t\geq 1
\end{align*} holds for some constant  $c \in (0,\infty)$.
\item[$(2)$] If  $\sigma$ is constant, then there exists a constant $C \in (0,\infty)$   such that  when $\eta\le \eta_0$,  $\bar\mu$   satisfies the log-Sobolev inequality
\beq\label{LS} \bar\mu(f^2\log f^2) \le C  \bar\mu(|\nn f|^2),\ \ \ f\in C_b^1(\R^d),\ \bar\mu(f^2)=1.\end{equation}
Moreover, there exists a  constant $\bar{\eta}_0\in (0,\eta_0]$ depending only on $(a,b^{(1)},k,d)$    such that  when $\eta\in(0,\bar{\eta}_0]$,
\begin{equation*}\beg{split}&\W_2(P_t^\ast\mu,\bar{\mu})^2+ \mathrm{Ent}(P_t^\ast\mu|\bar{\mu})\\
&\leq c \e^{-2\lambda t}\min\big\{\W_2(\mu,\bar\mu)^2,\mathrm{Ent}(\mu|\bar{\mu})\big\},\ \ \mu\in\scr P_2,t\geq 1\end{split}
\end{equation*}  holds for some constant  $c\in (0,\infty)$.
\end{enumerate}
\end{thm}

\beg{rem}  When $\sigma$ depends on the spatial variable, even in the case $b^{(0)}=0$, it is open whether $(C_1)$-$(C_2)$ is sufficient to ensure the log-Sobolev inequality for the invariant probability measure.
To establish the log-Sobolev inequality for non-constant $\si$, one may apply  Bakry-\'{E}mery's  curvature condition \cite{BE} which is less explicit in higher dimensions, or apply the dimension-free Harnack inequality
\cite{W11}.  Since these conditions are less explicit and hard to be modified to McKean-Vlasov SDEs, here we simply consider constant $\si$.

\end{rem}

To prove Theorem \ref{T0z}, a key step is to show the Lipschitz continuity of
$b^{(0)}(\cdot,\bar{\mu}).$  To this end,  we need to characterize the regularity of $\bar\rr(x)=\ff{\bar\mu(\d x)}{\d x}.$ Since $\bar\mu$ is $P_t^{\bar\mu}$-invariant, and by Duhamel's formula
\eqref{DH} the semigroup $P_t^{\bar\mu}$ can be represented by using $\hat P_t$, which is the diffusion semigroup generated by  $\hat L$   in \eqref{HTL}.
So,  we first study the heat kernel $\hat{p}_t(x,y) $ of  $\hat P_t$, which is the probability density function such that
$$\hat P_tf(x):=\int_{\R^d} f(y)\hat p_t(x,y)\d y,\ \ \ t>0,\ x\in \R^d,\ f\in\B_b(\R^d).$$
According to  {\cite[Theorem 1.2(i)]{MPZ}}, $\hat p_t$ is comparable with the Gaussian heat kernel
$$g_\kappa (t,x,y)=(\kappa^{-1}\pi t)^{-\frac{d}{2}}\exp\left[\frac{-\kappa |y-\psi_t(x)|^2}{ t}\right],\ \ t>0,\ x,y\in\R^d$$
for some constant $\kk\in (0,\infty)$, where  $\psi_t(x)$ is the unique solution to the ODE
$$\d \psi_t(x)=b^{(1)}(\psi_t(x))\d t, \ \ \psi_0(x)=x.$$ Let
\beg{align*} P_t^\kappa f(x):=\int_{\R^d}g_\kappa (t,x,y)f(y)\d y,\ \ f\in\scr B_b(\R^d),\ x\in\R^d,\ t> 0.\end{align*}   By the proof of \cite[Lemma 5.3]{HRW25}, there exists a constant
$c_1\in (0,\infty)$ such that
\beq\label{YS1} \|P_t^\kappa\|_{\tt L^{p_1}\to\tt L^{p_2}}\le c_1 t^{-\ff {d(p_2-p_1)}{2p_1p_2}},\ \ t\in (0,1],\ \ 1\le p_1\le p_2\le \infty.\end{equation}
We have the following result, where $\nn_x^0$ is the identity operator and $\nn_x^1$ is the gradient operator in $x\in\R^d$.

 \begin{lem}\label{hap} Assume {\bf (C)}. Then   there exist constants $c,\kk>0$ such that
$$ \big|\nn_x^i\hat{p}_t(x,y_1)-\nn_x^i\hat{p}_t(x,y_2)\big|\leq  c\frac{|y_1-y_2|^\aa}{t^{\frac{\aa+i}2}}\big(g_\kk (t,x,y_1)+g_\kk (t,x,y_2)\big),\ \ i=0,1$$ holds for any
$ t,\aa\in(0,1]$ and $x,y_1,y_2\in\R^d.$
 \end{lem}
\begin{proof} The desired estimate for $i=1$ is included in \cite[Lemma A.1(C1)]{MPZ}, while that for $i=0$ follows from    \cite[Theorem 1.2(i),(iv)]{MPZ}, which ensures
\begin{align}\label{gry}
|\nabla_y^i\hat{p}_t(x,y)|\leq ct^{-\frac{i}{2}}g_{2\kk}(t,x,y),\ \ i=0,1,\ \ t\in(0,1],\ x,y\in\R^d
\end{align} for some constants $c,\kk\in (0,\infty).$ Indeed, \eqref{gry} with $i=1$
implies that for some constant $c_1\in [c,\infty)$  \begin{align*}
&|\hat{p}_t(x,y_1)-\hat{p}_t(x,y_2)|=\left|\int_0^1\<y_1-y_2, \nabla_y\hat{p}_t(x,y_2+s(y_1-y_2))\>\d s\right|\\
&\leq c\frac{|y_1-y_2|}{t^{\frac{1}{2}}}\int_0^1 g_{2\kk}(t,x,y_2+s(y_1-y_2))\d s\\
&\leq  \frac{c_1|y_1-y_2|}{t^{\frac{1}{2}}} g_\kk(t,x,y_2),\ \ \text{if}\ |y_1-y_2|^2\le t,
\end{align*} where the last step follows from the fact
\beg{align*}& \ff {2\kk} t |\psi_t(x)- sy_1-(1-s)y_2|^2 \\
&= \ff {2\kk} t\Big(|\psi_t(x)-y_2|^2 + 2\<\psi_t(x)-y_2,\ s(y_2-y_1)\big\>+s^2|y_2-y_1|^2\Big)\\
&\ge \ff \kk t |\psi_t(x)-y_2|^2 -  \ff{\kk s^2}t |y_2-y_1|^2 \\
&\ge \ff \kk t |\psi_t(x)-y_2|^2 - \kk,\ \ \ s\in [0,1],\ |y_1-y_2|^2\le t.\end{align*}
On the other hand, by \eqref{gry} for $i=0$, we find a constant $c_2\in [c,\infty)$ such that
\beq\label{LOI} \beg{split} |\hat{p}_t(x,y_1)-\hat{p}_t(x,y_2)|&\leq c(g_{2\kk}(t,x,y_1)+g_{2\kk}(t,x,y_2))\\
&\le c_2 (g_{\kk}(t,x,y_1)+g_{\kk}(t,x,y_2)\big),\end{split}\end{equation}
  so that  for $|y_1-y_2|^2>t,$
\beq\label{LX1} |\hat{p}_t(x,y_1)-\hat{p}_t(x,y_2)|\leq \ff{c_2|y_1-y_2|} {t^{\ff 1 2}} (g_\kk(t,x,y_1)+g_\kk(t,x,y_2)).\end{equation}
In conclusion, \eqref{LX1}  holds in any case, which together with \eqref{LOI} yields  \beg{align*} &|\hat{p}_t(x,y_1)-\hat{p}_t(x,y_2)|=|\hat{p}_t(x,y_1)-\hat{p}_t(x,y_2)|^\aa|\hat{p}_t(x,y_1)-\hat{p}_t(x,y_2)|^{1-\aa} \\
&\le \Big(\frac{c_1|y_1-y_2|}{t^{\frac{1}{2}}}\big(g_{\kk}(t,x,y_1)+g_\kk(t,y_2)\Big)^\aa\Big(c_2(g_\kk(t,x,y_1)+g_\kk(t,x,y_2))\Big)^{1-\aa}\\
& \le \frac{(c_1\vee c_2)|y_1-y_2|^\aa}{t^{\frac{\aa}2}}\big(g_\kk (t,x,y_1)+g_\kk (t,x,y_2)\big),\ \ \ t,\aa\in (0,1],\ x,y_1,y_2\in\R^d.\end{align*}
\end{proof}

\begin{lem}\label{regb0}
Assume {\bf (C)} and let $\eta_0$ be in Theorem \ref{T0}. Then there exists a constant  $c\in (0,\infty)$
such that when $\eta\in (0,\eta_0]$,
\beq\label{hob} |b^{(0)}(x,\bar{\mu})-b^{(0)}(y,\bar{\mu})|\leq c\eta  (1\land  |x-y|),   \ \ x,y\in\R^d.
\end{equation}
 \end{lem}
\begin{proof} Recall $\bar\rr(x)=\ff{\bar\mu(\d x)}{\d x}.$ Since {\bf (C)} implies {\bf (A)} and {\bf (B)} for  $\eta:=\|h\|_{\tt L^k}\le \eta_0$,
 by    Theorem \ref{T0}  we find a constant $c_0\in (0,\infty)$ such that
\beq\label{CB}\sup_{\eta\le \eta_0} \|\bar\rr\|_\infty\le c_0,\ \ |b^{(0)}(x,\bar{\mu})| \le \eta_0 \|\bar\mu\|_{k*}\le c_0.\end{equation}
Since $\bar\mu$ is the invariant probability measure for the diffusion semigroup $P_t^{\bar\mu}$, we have
\beq\label{OU2} \bar \rr(y)= \int_{\R^d} \bar\rr(z)\bar p_t(z,y)\d z,\ \ \ t>0,\ y\in\R^d,\end{equation}
where $\bar p_t$ is the heat kernel of $P_t^{\bar\mu}.$
On the other hand, by  Duhamel's formula \eqref{DH} and
$$\nn \hat P_{t}f(x)= \nn_x\int_{\R^d} \hat p_t(x,y)f(y)\d y= \int_{\R^d} \big(\nn_x\hat p_{t}(x,y) \big) f(y)\d y,\ \ t\in(0,1],f\in\scr B_b(\R^d)$$   due to \eqref{gry},
we derive
 {\beg{align*} &\int_{\R^d} \bar p_t(z,y)f(y)\d y=  \int_{\R^d} \hat p_t(z,y)f(y)\d y +\int_0^t \int_{\R^d} \bar p_s(z,x)  \big\<b^{(0)}(x,\bar\mu), \nn \hat P_{t-s} f(x)\big\>\d x\d s  \\
&= \int_{\R^d}f(y) \bigg( \hat p_t(z,y) + \int_0^t \int_{\R^d} \bar p_s(z,x) \big\<b^{(0)}(x,\bar\mu), \nn_x \hat p_{t-s} (x,y)\big\>\d x\d s\bigg)\d y,\ \ f\in \B_b(\R^d).\end{align*}
So,
$$\bar p_t(z,y) = \hat p_t(z,y)+ \int_0^t \int_{\R^d} \bar p_s(z,x) \big\<b^{(0)}(x,\bar\mu), \nn_x \hat p_{t-s} (x,y)\big\>\d x\d s\bigg),\ \ \ t\in(0,1],\ x,y\in\R^d.$$}
Combining this with \eqref{OU2}, we obtain
 \begin{align*}
\bar{\rho}(y)&=\int_{\R^d}\bar{\rho}(x)\hat{p}_1(x,y)\d x+\int_0^1\d s\int_{\R^d}\bar{\rho}(x)\<b^{(0)}(x,\bar{\mu}),\nabla_x\hat{p}_{1-s}(x,y)\>\d x\\
&=\int_{\R^d}\bar{\rho}(x)\hat{p}_1(x,y)\d x+\int_0^1\d s \int_{\R^d}\bar{\rho}(x)\<b^{(0)}(x,\bar{\mu}),\nabla_x\hat{p}_{s}(x,y)\>\d x.
\end{align*}
Combining these with Lemma \ref{hap} and \eqref{CB},  we find  a constant  $c_1 \in (0,\infty)$ such that
\begin{align*}
&|\bar{\rho}(y_1)-\bar{\rho}(y_2)|\\
&\leq\int_{\R^d}\bar{\rho}(x)|\hat{p}_1(x,y_1)-\hat{p}_1(x,y_2)|\d x\\
&+\int_0^1\d s \int_{\R^d}\bar{\rho}(x)|b^{(0)}(x,\bar{\mu})||\nabla_x\hat{p}_{s}(x,y_1)-\nabla_x\hat{p}_{s}(x,y_2)| \d x\d s\\
&\le c_1 |y_1-y_2|^\aa\int_{\R^d}  \big(g_\kk (1,x,y_1)+g_\kk (1,x,y_2)\big)\bar\mu(\d x)\\
&\quad +c_1 |y_1-y_2|^\aa\int_0^1 s^{-\frac{1+\aa}{2}}\d s \int_{\R^d} \big(g_\kk (s,x,y_1)+g_\kk (s,x,y_2)\big)\bar\mu(\d x).
\end{align*}
Consequently,
  \begin{align*}
&|b^{(0)}(y_1,\bar{\mu})-b^{(0)}(y_2,\bar{\mu})|\le
  \int_{\R^d}|h|(z)|\bar{\rho}(y_1-z)-\bar{\rho}(y_2-z)|\d z\\
&\leq c_1 |y_1-y_2|^\aa\int_{\R^d\times\R^d}|h|(z) \{g_\kk (1,x,y_1-z)+g_\kk (1,x,y_2-z)\}\bar\mu(\d x)\d z\\
&+c_1 |y_1-y_2|^\aa\int_0^1  s^{-\frac{1+\aa}{2}}\d s \int_{\R^d\times\R^d }|h|(z)  \{g_\kk (s,x,y_1-z)+g_\kk (s,x,y_2-z)\}\bar\mu(\d x) \d z\\
& =   c_1 |y_1-y_2|^\aa\bar\mu\big( P_1^\kk[ |h(y_1-\cdot)|  + |h(y_2-\cdot)|] \big)\\
& \qquad +\int_0^1 s^{-\ff{1+\aa}2}\bar\mu\big( P_s^\kk |h(y_1-\cdot)|  + P_s^\kk  |h(y_2-\cdot)| \big) \d s\\
&\leq 2c_1 |y_1-y_2|^\aa \|\bar\mu\|_{k*} \|h\|_{\tt L^k} \bigg(\|P_1^\kk\|_{\tt L^k\to\tt L^k}  +\int_0^1 s^{-\ff{1+\aa}2}\| P_s^\kk\|_{\tt L^k\to\tt L^k} \d s\bigg).
\end{align*}
Combining this with  $\|h\|_{\tt L^k} =\eta\le\eta_0$, \eqref{BM} and \eqref{YS1} for $p_1=p_2=k$, we find a constant $c_2\in (0,\infty)$ such that
\begin{align}\label{hob'}|b^{(0)}(x,\bar{\mu})-b^{(0)}(y,\bar{\mu})|\leq \ff{c_2   |y_1-y_2|^\aa}{1-\aa},   \ \ x,y\in\R^d,\ \aa\in (0,1), \eta\in(0,\eta_0].
\end{align}

 (2) Since $b^{(1)}$ is Lipschitz continuous ensured by {\bf (C)},    \eqref{hob'} implies that
 $$\bar b(x):= b^{(0)}(x,\bar\mu)+ b^{(1)}(x)$$ satisfies
  \begin{align*}
 |\bar b (x)-\bar b(y)|
 \leq \ff{c}{1-\aa}  (|x-y|^\aa\vee|x-y|),\ \ \eta\in(0,\eta_0],\aa\in(0,1).
\end{align*}
According to \cite[Theorem 1.2(iv)]{MPZ}, this together with the conditions on $\si$ included in {\bf (C)}  implies
\begin{align}\label{nay}|\nabla_y p_1^{\bar{\mu}}(x,y)|\leq c_3  g_{\kk}(1,x,y),\ \ \eta\in(0,\eta_0],\ x,y\in\R^d
\end{align}
for some constants $c_3,\kk \in (0,\infty)$.
By $P_1^{\bar{\mu}*}\bar{\mu}=\bar{\mu}$ we have
\begin{align*}
b^{(0)}(x,\bar{\mu})=b^{(0)}(x,P_1^{\bar{\mu}*}\bar{\mu})&=\int_{\R^d\times\R^d} h(x-z)p_1^{\bar{\mu}}(y,z)\d z\bar{\mu}(\d y)\\
&=\int_{\R^d\times\R^d} h(z)p_1^{\bar{\mu}}(y,x-z)\d z\bar{\mu}(\d y).
\end{align*}
Combining this with  $\|h\|_{\tt L^k} =\eta\le\eta_0$, \eqref{BM}, \eqref{YS1} for $p_1=p_2=k$ and \eqref{nay}, we find a constant $c \in (0,\infty)$ such that
$$ |b^{(0)}(\cdot,\bar\mu)|\le \|h\|_{\tt L^k}\|\bar\mu\|_{k*}\le c\eta,\ \ \ \eta\in [0,\eta_0]$$ and
\begin{align*}|\nabla b^{(0)}(\cdot,\bar{\mu})(x)|&\leq \left|\int_{\R^d}\int_{\R^d}h(z)\nabla p_1^{\bar{\mu}}(y,\cdot)(x-z)\d z\bar{\mu}(\d y)\right|\\
&\leq c_3  \int_{\R^d}\int_{\R^d}|h|(z) g_{\kk}(1,y,x-z)\d z\bar{\mu}(\d y) \\
&= c_3 \int_{\R^d}\int_{\R^d}|h|(x-z) g_{\kk}(1,y,z)\d z\bar{\mu}(\d y) \\
&\leq c_3 \|\bar{\mu}\|_{k*}\|h\|_{\tilde{L}^k}\|P_1^\kk\|_{\tt L^k\to\tt L^k} \leq c\eta,\ \ \eta\in [0,\eta_0].
\end{align*}
 Therefore,  we finish the proof.
\end{proof}

We are now ready to prove Theorem \ref{T0z} in two different situations.

\beg{proof}[Proof of Theorem $\ref{T0z}(1)$] Let $\mu\in \scr P_2$ and $R=0$.
Take $X_0, Y_0$ be $\F_0$-measurable such that
\beq\label{KO1} \L_{X_0}=\mu, \ \ \L_{Y_0}=\bar{\mu},\ \ \W_2(\mu,\bar{\mu})^2=\E|X_0-Y_0|^2,\end{equation}
and let $X_t,Y_t$ solve the SDEs with initial values $X_0$ and $Y_0$ respectively:
\beq\label{CPL} \begin{split}
&\d X_t=\big\{b^{(1)}(X_t)+b^{(0)}(X_t,P_t^\ast\mu)\big\}\d t+\sigma(X_t)\d W_t,\\
&\d Y_t  = \big\{b^{(1)}(Y_t)+b^{(0)}(Y_t,\bar\mu)\big\}\d t+\sigma(Y_t) \d W_t.\end{split}\end{equation}
Since $\L_{X_0}=\mu$, and $\L_{Y_0}=\bar\mu$ is $P_t^*$-invariant, we have
\beq\label{KO2} P_t^*\mu=\L_{X_t},\ \ \ \bar\mu= \L_{Y_t},\ \ \ t\ge 0.\end{equation}
 By \eqref{hob} and {\bf(C)} with $R=0$, we obtain
\begin{align*}&\|\si(x)-\si(y)\|_{HS}^2 + 2 \big\< b^{(1)}(x)+ b^{(0)}(x,P_t^*\mu)- b^{(1)}(y)- b^{(0)}(y,\bar\mu),\ x-y\big\>\\
&\le -2 K |x-y|^2 + 2 |b^{(0)}(x,P_t^\ast\mu)-b^{(0)}(y,\bar{\mu})|\cdot |x-y|\\
&\le  (\eta -2 K) |x-y|^2  + 2\eta\|P_t^\ast\mu-\bar{\mu}\|_{k*}|x-y|,   \ \ \ x,y\in\R^d.
\end{align*}
So, by \eqref{CPL} and  It\^{o}'s formula, we derive
$$
 \d |X_t-Y_t|^2\leq (\eta-2K) |X_t -Y_t|^2\d t  + 2\eta\|P_t^\ast\mu-\bar{\mu}\|_{k*}|X_t-Y_t| \d t+\d M_t
$$
for some local martingale $M_t$. Since $\L_{Y_t}=\bar\mu\in \scr P_2$ and $\L_{X_0}=\mu\in \scr P_2,$ this implies
$$ \sup_{s\in [0,t]} \E[|X_s-Y_s|^2]<\infty,\ \ \ t\in (0,\infty),$$
and by combining   with \eqref{KO1} we obtain
  \beg{align*}& \sup_{s\in [0,t]} \e^{(2K-\eta)s} \E[|X_s-Y_s|^2]\le \W_2(\mu,\bar\mu)^2 +2\eta \int_0^t \e^{(2K-\eta) s}\|P_s^*\mu-\bar\mu\|_{k*} \E[|X_s-Y_s|]\d s\\
&\le \W_2(\mu,\bar\mu)^2 +  \eta   \sup_{s\in [0,t]} \e^{(2K-\eta) s} \E[|X_s-Y_s|^2]+  \eta  \bigg(\int_0^t \e^{(K-\eta/2) s} \|P_s^*\mu-\bar\mu\|_{k*} \d s\bigg)^{ 2}.\end{align*}
So,  when  $\eta\in [0, 1\wedge(2K))$, we have
\beq\label{KO3}  { \E[|X_t-Y_t|^2]\le \ff {\e^{-(2K-\eta) t}} {1-\eta} \W_2(\mu,\bar\mu)^2+ \ff{\eta}{1-\eta}   \bigg(\int_0^t  \|P_s^*\mu-\bar\mu\|_{k*} \d s\bigg)^{ 2}.}\end{equation}
Noting that {\bf (C)} implies the assumption \cite[{\bf (A)}]{HRW25} with constant $K$ uniformly in $\|h\|_{\tt L^k}=\eta\le \eta_0$,
so that  \cite[(2.16) and (2.18)]{HRW25} hold  for $p=\infty,q=1$ and $\tau(\gg)=\infty$ for any $\gg\in \scr P$. Consequently, there exist  a constant $c_1\in (0,\infty)$
such that when  $\eta\le \eta_0,$
\beq\label{KO4} \begin{split}
&\|P_t^*\gg-\bar{\mu}\|_{k*}\le c_1  t^{-\ff 1 2-\ff d{2k}}\W_2(\mu,\bar\mu),\\
&\Ent(P_t^*\mu|\bar\mu)\le c_1 t^{-1} \W_2(\mu,\bar\mu)^2, \ \ \mu \in \scr P_2,\ \ t\in(0,1].
\end{split} \end{equation}
Combining the first inequality with  \eqref{KO2} and \eqref{KO3}, we find a constant $c_2\in (0,\infty)$ such that
\beg{align*}  \W_2(P_t^*\mu,\bar\mu)^2&\le \E[|X_t-Y_t|^2] \le   \W_2(\mu,\bar\mu)^2\left[\ff {\e^{-(2K-\eta)t} }{1-\eta}+ \ff{\eta }{1-\eta}
\bigg(\int_0^t  c_1 s^{-\ff 1 2-\ff d{2k}}  \d s\bigg)^{ 2}\right]\\
&\le \W_2(\mu,\bar\mu)^2\bigg(\ff {\e^{-(2K-\eta)t} }{1-\eta}+ \ff{c_2t^{\ff{k-d}k}\eta }{1-\eta}\bigg),\ \ t\ge 0,\  \eta\in [0,\eta_0\land  (2K)\land 1].\end{align*}
Taking $\bar\eta_0\in (0, \eta_0\land K\land 1)$ such that
$$\ff{\e^{-K}}{1-\bar \eta_0}+ \ff{c_2\bar\eta_0}{1-\bar\eta_0}<1,$$
and applying the semigroup property of $P_t^*$,
we  find   constants $c,\ll>0$ such that
$$\W_2(P_t^*\mu,\bar\mu)\le c \e^{-\ll t} \W_2(\mu,\bar\mu),\ \ \ t\ge 0.$$
 Combining this with the second inequality in \eqref{KO4}, we finish the proof.
 \end{proof}

\begin{proof} [Proof of Theorem $\ref{T0z}(2)$]  Let {\bf (C)} hold with constant $\si$.  We emphasize   that   all constants $(c_i)_{i\ge 1}$ and $t_0$ below are uniformly in $\eta\in [0, \eta_0]$, although
 the invariant probability measure  $\bar\mu$  depends on $\eta=\|h\|_{\tt L^k}$.

(a)   By   {\bf (C)}  and \eqref{bob},  we find a constant $c_1 \in (0,\infty)$  such that
 \begin{align*}&\<(b^{(1)}+b^{(0)}(\cdot,\bar{\mu}))(x)-(b^{(1)}+b^{(0)}(\cdot,\bar{\mu}))(y),\ x-y\> \\
 &\qquad \qquad\qquad \leq c_1 |x-y| -K|x-y|^2,\ \ x,y\in \R^d,\\
&  2\<b^{(1)}(x)+b^{(0)}(x,\bar{\mu}),x\>+\|\sigma\|_{HS}^2\leq c_1  -K|x|^2,\ \ \ x\in\R^d.\end{align*}
So, by the proof of  \cite[Theorem 2.6(1)]{HKR25}, we find   constants  $t_0,c_2\in (0,\infty)$   such that
\begin{align}\label{hypbo}
  \|P_{t_0}^{\bar{\mu}}\|_{L^2(\bar{\mu})\to L^{4}(\bar{\mu})} \le c_2.
\end{align}   Moreover, by {\bf (C)} and \eqref{hob}, we find a constant $c_3\in (0,\infty)$ such that
\begin{align*}|(b^{(0)}(\cdot,\bar{\mu})+b^{(1)})(x)-(b^{(0)}(\cdot,\bar{\mu})+b^{(1)})(y)|\leq c_3 |x-y|,\ \ x,y\in \R^d.
\end{align*}
This together with $\nn\si=0$  implies
 $$|\nabla P_t^{\bar{\mu}}f|\leq \e^{c_3 t}P_t^{\bar{\mu}}|\nabla f|,\ \ f\in C_b^1(\R^d).$$
Combining this with $\|a\|_\infty+\|a^{-1}\|_\infty<\infty$, by the standard Bakry-\'{E}mery argument \cite{BE}, see for instance the proof of \cite[(2.4)]{HKR25}, we  find a constant $c_4\in (0,\infty)$
such that the following semigroup log-Sobolev inequality
holds: $$  P_{t}^{\bar{\mu}}(f^2\log f^2)-(P_{t}^{\bar{\mu}}f^2)\log P_{t}^{\bar{\mu}}f^2  \leq  c_4(\e^{c_4t}-1)  P_{t}^{\bar{\mu}}  |\nabla   f |^2,
 \ \ t\geq 0,   f\in C_b^1(\R^d).
$$
According to \cite[Lemma 2.5(1)]{HKR25}, this together with \eqref{hypbo} implies   the defective log-Sobolev inequality  for some constant $c_5\in (0,\infty)$:
 \beq \label{delos}\bar{\mu}(f^2\log f^2)\leq c_5\bar{\mu}(|\nabla f|^2)+c_5,\ \ \ f\in C_b^\infty(\R^d), \bar{\mu}(f^2)=1,
\end{equation} which further implies the defective Poincar\'e inequality for possibly different constant $c_5\in (0,\infty)$:
\beq\label{PC} \bar\mu(f^2)\le c_5\bar\mu(|\nn f|^2) + c_5\mu(|f|)^2,\ \ \ f\in C_b^1(\R^d).\end{equation}
 Moreover, according to \cite[Corollary 1.2.11]{BKR}, under {\bf (C)} we have $\bar\rr=\e^{V}$ for some locally bounded function $V$ on $\R^d$ uniformly in $\eta\in [0,\eta_0]$.
 Combining this with the classical Poincar\'e inequality on balls,  we may apply  \cite[Theorem 3.1]{RW01} to derive the weak Poincar\'e inequality
 $$\bar\mu(f^2)\le \aa(r) \bar\mu(|\nn f|^2)+ r\|f\|_\infty^2,\ \ \ r>0,\ f\in C_b^1(\R^d), \bar\mu(f)=0,$$
 where $\aa: (0,\infty)\to (0,\infty)$ is uniformly in $\eta\in [0,\eta_0]$. According to \cite[Proposition 4.1.2]{Wbook}, this together with \eqref{PC} implies the Poincar\'e inequality, so that
 \eqref{delos} implies the desired log-Sobolev inequality \eqref{LS} for some constant $C\in (0,\infty)$ uniformly in $\eta\in [0,\eta_0].$

(b) According to \cite{BGL},  \eqref{LS} implies the Talagrand inequality
\begin{align}\label{talgr}\W_2(\mu,\bar{\mu})^2\leq  {C}\,\mathrm{Ent}(\mu|\bar{\mu}),\ \ \mu\in\scr P_2.
\end{align}
Moreover, since $\|\si^{-1}\|_{\infty}<\infty$,  \eqref{LS} implies the log-Sobolev inequality for the diffusion process with semigroup $P_t^{\bar \mu}:$
$$\bar\mu(f^2\log f^2)\le C \|\si^{-1}\|_\infty^2 \bar\mu(\si^*\nn f|^2),\ \ \ f\in C_b^1(\R^d),\ \bar\mu(f^2)=1.$$
It is well-known that this implies the exponential ergodicity of $P_t^{\bar\mu*}$ in entropy:
\beq\label{ENT} \Ent(P_t^{\bar\mu*}\mu|\bar\mu)\le \e^{-2\ll_1 t} \Ent(\mu|\bar\mu),\ \ \mu\in \scr P,\ t\ge 0,\end{equation} where
$\ll_1:= \ff{2}{C\|\si^{-1}\|_\infty^2}\in (0,\infty),$ see for instance \cite[Theorem 5.2.1]{BL}.
By {\bf (C)}, since $\si\si^*$ is constant and invertible,  $b^{(1)}$ is Lipschitz continuous  and $b^{(0)}(\cdot,\bar\mu)$ is bounded,
by \cite[Lemma 4.2]{YZ} we find a constant $c_6\in (0,\infty)$ such that
$$\Ent(P_t^{\bar\mu*}\mu|\bar\mu) = \Ent(P_t^{\bar\mu*}\mu|P_t^*\bar\mu)\le \ff {c_6}{1\land t } \W_2(\mu,\bar\mu)^2,\ \ t>0,\ \mu\in \scr P_2.$$
Combining this with \eqref{talgr} and \eqref{ENT} we find a constant $c_7\in (0,\infty)$ such that
  \beg{align*} &\W_2(P_t^{\bar\mu*}\mu,\bar{\mu})^2 \le C \Ent(P_t^{\bar\mu*}\mu| \bar\mu)\\
&\le C\e^{-2\ll_1 (t-1)} \Ent(P_1^{\bar\mu*}\mu| \bar\mu)\le c_7^2 \e^{-2\ll_1 t} \W_2(\mu,\bar\mu)^2,\ \ t\ge 1.\end{align*}
This together with  $P_t^*\mu= \bar{P}_t^{\mu*}\mu$ due to \eqref{LX} and the triangle inequality yields
\beq\label{TRA} \beg{split}&\W_2(P_t^*\mu, \bar\mu)\le \W_2(P_t^{\bar \mu*}\mu, \bar\mu)+ \W_2(P_t^{\bar \mu*}\mu, \bar{P}_t^{\mu*}\mu)\\
&\le c_7 \e^{-\ll_1 t} \W_2(\mu,\bar\mu)+\W_2(P_t^{\bar \mu*}\mu, P_t^{\mu*}\mu),\ \ t\ge 1.\end{split}\end{equation}

To estimate $\W_2(P_t^{\bar \mu*}\mu, \bar{P}_t^{\mu*}\mu)$, let $X_0=\bar X_0$ be $\F_0$-measurable with $\L_{X_0}=\mu$, and let $(X_t,\bar X_t)$ solve
 \begin{align*}
&\d X_t=\big\{b^{(1)}(X_t)+b^{(0)}(X_t,P_t^*\mu)\big\}\d t+\sigma(X_t)\d W_t,\\
&\d \bar X_t  = \big\{b^{(1)}(\bar X_t)+b^{(0)}(\bar X_t,\bar\mu)\big\}\d t+\sigma(\bar X_t) \d W_t.\end{align*}
Then $\L_{X_t}= P_t^*\mu, \L_{\bar X_t}=P_t^{\bar\mu*}\mu$, so that
\beq\label{ER} \W_2(P_t^*\mu,P_t^{\bar\mu*} \mu)^2\le \E[|X_t-\bar X_t|^2],\ \ \ t\ge 0.\end{equation}
 {By    {\bf (C)} and \eqref{hob}, we find a constant $c_6\in (0,\infty)$ such that
$$2\<b(x,\mu)-b(y,\bar{\mu}), x-y\>\le c_6 |x-y|^2+ 2\eta  |x-y|\|\mu-\bar{\mu}\|_{k*},$$}
so that by It\^o's formula,
 \begin{align*}
\d |X_t-\bar{X}_t|^2\leq c_6 |X_t-\bar{X}_t|^2\d t+2\eta\|P_t^\ast\mu-\bar{\mu}\|_{k*}|X_t-\bar{X}_t|\d t.
\end{align*}
This and $X_0=\bar X_0$ imply
\begin{align*}
&\sup_{s\in[0,t]}|X_s-\bar{X}_s|^2\leq 2 \eta \e^{c_6t}\int_0^t  \|P_s^\ast\mu-\bar{\mu}\|_{k*}|X_s-\bar{X}_s|\d s\\
&\leq \frac{1}{2} \sup_{s\in[0,t]}|X_s-\bar{X}_r|^2+2\eta^2 \left(\e^{c_6t} \int_0^t  \|P_s^\ast\mu-\bar{\mu}\|_{k*}\d s\right)^2.
\end{align*}
This implies
$$\sup_{s\in[0,t]}|X_s-\bar{X}_s|^2\le 4 \eta^2 \left(\e^{c_6t} \int_0^t  \|P_s^\ast\mu-\bar{\mu}\|_{k*}\d s\right)^2,$$
so that \eqref{ER} and the first inequality in \eqref{KO4}  yield
\begin{align*}
\W_2(P_t^\ast\mu,P_t^{\bar{\mu}\ast}\mu)^2&\leq \E|X_t^\mu-\bar{X}_t|^2\leq \left(2\eta \e^{c_6t} \int_0^t \|P_s^\ast\mu-\bar{\mu}\|_{k*}\d s\right)^2\\
&\le \big(c_8 \eta \e^{c_8t} \W_2(\mu,\bar\mu)\big)^2,\ \ t\ge 0,\ \mu\in \scr P_2.
\end{align*}
Combining with   \eqref{TRA} gives
$$ \W_2(P_t^*\mu,\bar\mu)    \le  \W_2(\mu,\bar\mu)\Big( {c_7 \e^{-\ll_1 t}}+ c_8 \eta \e^{c_8t} \Big) ,\ \ t\ge 1,\ \mu\in\scr P_2.$$
Taking $t_0\in [1,\infty)$ and $\bar\eta_0\in (0,\eta_0]$ such that  when $\eta\le \bar\eta_0$,
$$ {c_7 \e^{-\ll_1 t_0}}+ c_8\bar\eta_0 \e^{c_8t_0} <1,$$
as in step (a), we find  constants $c',\ll'\in (0,\infty)$ such that
$$\W_2(P_t^*\mu,\bar\mu) \le c'\e^{-\ll' t} \W_2(\mu,\bar\mu),\ \ \ t\ge 0,\ \mu\in \scr P_2.$$
According to  \cite[Theorem 2.1]{RW21},  the proof is finished by combining this   with \eqref{talgr} and  {the second inequality in \eqref{KO4}}.  \end{proof}

 {\bf Acknowledgement.} The   authors would like to thank the referee for helpful comments.

\end{document}